\documentclass[12pt]{article}

\usepackage{fullpage}
\usepackage{amssymb}
\usepackage{amsmath}
\usepackage{amsthm}
\usepackage{url}
\usepackage{amsrefs}
\usepackage{latexsym}
\usepackage{graphicx}
\usepackage[noadjust]{cite}

\newtheorem{theorem}{Theorem}
\newtheorem{corollary}{Corollary}

\newtheorem{lemma}{Lemma}

\newtheorem{definition}{Definition}

\def\eop{\hfill$\Box$}

\def\e{{\bf 1}}

\begin{document}

\title{Meertens number and its variations}

\author{Chai Wah Wu\\ IBM T. J. Watson Research Center\\ P. O. Box 218, Yorktown Heights, New York 10598, USA\\e-mail: chaiwahwu@member.ams.org}
\date{March 28, 2016}
\maketitle

\begin{abstract}
In 1998, Bird introduced Meertens numbers as numbers that are invariant under a map similar to the G\"{o}del encoding.  In base 10, the only known Meertens number is $81312000$.  We look at some properties of Meertens numbers and consider variations of this concept.  In particular, we consider variations where there is a finite time algorithm to decide whether such numbers exist.
\end{abstract}

\section{Introduction}
Kurt G\"{o}del in his celebrated work on mathematical logic \cite{godel:incompleteness:1931} uses an injective map from the set of finite sequences of symbols to the set of natural numbers in order to describe statements in logic as natural numbers and relating properties of mathematical proofs with properties of natural numbers. This approach is subsequently used by Alan Turing to define the notion of computable numbers \cite{turing:computable:1937}, which are numbers that can be computed by his abstract computing model.  This seminal work ushered in the field of theoretical computer science.  The basic G\"{o}del encoding is as follows: each symbol in an alphabet is mapped to a distinct positive integer.  Thus a finite sequence of symbols $s_1,\cdots ,s_n$ is mapped to a sequence of positive numbers $m_1,\cdots , m_n$.  This sequence is then mapped to a natural number $G = \prod_{i=1}^{n} p_i^{m_i}$, where $p_i$ is the $i$-th prime.   

In Ref. \cite{bird:meertens:1998}, Richard Bird dedicated the number $81312000$ to his friend Lambert Meertens on the occasion of his 25 years at the CWI institute and called it a {\em Meertens number}.  He constructed this number using a mapping similar to the G\"{o}del encoding.  

\begin{definition}
Given a decimal representation $d_1,\cdots ,d_n$ of the number $m = \sum_{i=1}^n d_i 10^{n-i}$, if $m = \prod_{i=1}^{n} p_i^{d_i}$ then $m$ is called a {\em Meertens number}.
\end{definition}

The only Meertens number known to date is $81312000 = 2^8 3^1 5^3 7^1 11^2 13^0 17^0 19^0$ \cite{bird:meertens:1998}.  David Applegate has conducted the search up to $10^{29}$
(see \url{https://oeis.org/A246532}) without finding any other Meertens number.

\begin{definition}
Let $b\geq 2$ and $0\leq d_i < b$ with $d_1 > 0$ be integers such that $m = \sum_{i=1}^n d_i b^{n-i}$, then $M_b(m)$ is defined as $M_b(m) =  \prod_{i=1}^{n} p_i^{d_i}$.
\end{definition}

Thus a Meertens number is a fixed point of the function $M_{10}(\cdot)$.
Note that the function $M_{10}$ is similar to the G\"{o}del encoding function.  However, unlike the G\"{o}del encoding, this function is not injective.  In particular, $M_{10}(10^k) = 2$ for all $k\geq 0$.
Since $d_i \leq 9$,  the exponent of the prime 2 and 5 must be less than or equal to 9 and thus a Meertens number has at most 9 trailing zeros.  In particular, the number of trailing zeros is the minimum of the first and third digit of $m$.

\section{Meertens number in other bases}
As noted in \cite{bird:meertens:1998}, the concept of a Meertens number can be defined in other number bases as well, i.e. $m$ is a Meertens number in base $b$ if  $m$ satisfies $m = \prod_{i=1}^{n} p_i^{d_i} = \sum_{i=1}^n d_i b^{n-i}$ for some nonnegative integers $d_i < b$ with $d_1 > 0$, i.e., $M_b(m) = m$.  Since $d_1\neq 0$, it is clear that a Meertens number must necessarily be even.  Similarly, the number of trailing zeros in base 10 is the minimum of the first and third digit of $m$ in base $b$.  
Table \ref{tbl:mn} lists some Meertens numbers found in various number bases.  

\begin{table}[htbp]
\begin{center}
\begin{tabular}{c|c}
\hline\hline
Number base & Meertens number \\
\hline
2 & 2, 6, 10 \\
3 & 10 \\
4 & 200 \\
5 & 6, 49000, 181500 \\
6 & 54 \\
7 & 100 \\
8 & 216 \\
9 & 4199040 \\
10 & 81312000 \\
14 & 47250 \\
16 & 18 \\
17 & 36 \\
19 & 96 \\
32 & 256 \\
51 & 54 \\
64 & 65536 \\
71 & 216 \\
158 & 162 \\
160 & 324 \\
323 & 1296 \\
481 & 486 \\
512 & 4294967296\\
1452 & 1458 \\
1455 & 2916 \\
1942 & 5832 \\
4096 & 65536 \\
4367 & 4374 \\
7775 & 46656 \\
8294 & 82944 \\
13114 & 13122 \\
13118 & 26244 \\
26242 & 104976 \\
39357 & 39366 \\
52485 & 157464 \\
74649 & 746496 \\
118088 & 118098 \\
209951 & 1679616 \\
354283 & 354294 \\
1062870 & 1062882 \\
1119743 & 10077696 
\end{tabular}
\end{center}
\caption{Meertens numbers in various number bases.}\label{tbl:mn}
\end{table}

The number $82944$ is interesting as it is a Meertens number in base $8294$ and shares the first $4$ digits with the base $8294$.  The number $82944$ in base $8294$ is $A4$ (where we borrow from hexadecimal notation and use $A$ to denote the digit $10$) and $2^{10}3^4 = 82944$.  Are there other numbers with this property?

\begin{theorem}
If $1024\cdot 3^c - c$ is divisible by $10$ for some integer $c\geq 0$, then $1024\cdot 3^c$ is a Meertens number in base $b= \frac{1024\cdot 3^c - c}{10}$.
\end{theorem}
\begin{proof} First note that $b > 10$, $b > c$ and $1024\cdot 3^c = 10b+c$ written in base $b$ has digits $10$ and $c$ which maps to $1024\cdot 3^c$ under the map $M$. \eop
\end{proof}
There are two solutions with $c < 10$, i.e, $c=4$ and $c=6$, with $c=4$ corresponding to the number $82944$ above and $c=6$ corresponding to a Meertens number $746496$ in base $74649$.  Similarly,
$2^{100}3^{96}-96$ is a Meertens number in base $\frac{2^{100}3^{96}-96}{100}$ and the base in decimal is equal to the Meertens number in decimal minus the last $2$ digits.

Since $M_b$ is not injective, it is possible for a number to be a Meertens number in more than one number base.  We note in Table \ref{tbl:mn} that $6$, $10$, $216$ and $65536$ are Meertens numbers in more than one number base.  Are there any others?  The answer is yes as a consequence of the following result.

\begin{theorem}
If $a$, $k$ and $m$ are positive numbers such that $a+km = 2^a$ and $a < k$, then $2^{2^{a}}$ is a Meertens number in base $2^k$.
In particular for $a>2$, $2^{2^{a}}$ is a Meertens number in base $2^{2^a-a}$.
\end{theorem}
\begin{proof}
Since $2^{2^{a}} = 2^a2^{km}$, this means that $2^{2^{a}}$ consists of a single digit of value $2^a < 2^k$ followed by $m$ zeros.
Thus $M_{2^k}\left(2^{2^{a}}\right) = 2^{2^{a}}$.  For $a>2$, $2^a-a > a$ and by setting $m=1$, this shows that $2^{2^a}$ is a Meertens number in base $2^{{2^a}-a}$. \eop
\end{proof}

In particular, we have the following Corollary:

\begin{corollary}
If $k>a$ is a divisor of $2^a-a$, then $2^{2^{a}}$ is a Meertens number in base $2^k$.
\end{corollary}
For small values of $a$ we list these divisors in Table \ref{tbl:m2}.

\begin{table}[htbp]
\begin{center}
\begin{tabular}{c|c}
\hline\hline
$a$ & $k$: divisors of $2^a-a$ larger than $a$ \\
\hline
3 & 5 \\
4 & 6, 12 \\
5 & 9, 27 \\
6 & 29, 58 \\
7 & 11, 121 \\
8 & 31, 62, 124, 248 \\
9 & 503 \\
10 & 13, 26, 39, 78, 169, 338, 507, 1014 \\
11 & 21, 97, 291, 679, 2037 \\
12 & 1021, 2042, 4084 \\
13 & 8179 \\
14 & 1637, 3274, 8185, 16370 \\
15 & 4679, 32753 
\end{tabular}
\end{center}
\caption{Values of $a$ and $k$ such that  $2^{2^{a}}$ is a Meertens number in base $2^k$.}\label{tbl:m2}
\end{table}

This shows that there are many numbers (for example $4294967296 = 2^{2^5}$) that are Meertens numbers in more than one base.  For instance $2^{2^{16}}$ is a Meertens number in at least $105$ different bases and $2^{2^{64}}$ is a Meertens number in at least $435$ bases!
In particular, for any integer $t > 2$, $2^{2^t-k}-2^{t-k}$ is a divisor of $2^{2^t}-2^t$ for $k=0,\cdots , t$.  Thus
$2^{2^{2^t}}$ is a Meertens number in at least $t+1$ different bases, i.e. there are numbers which are Meertens numbers for an arbitrarily large number of bases.
Even though there is only one known Meertens number in base 10, the above argument also implies that there are arbitrarily large bases for which Meertens numbers exist.

\begin{theorem}\label{thm:23}
For integers $m\geq n\geq 0$, 
\begin{itemize}
\item $2\cdot 3^n$ is a Meertens number in base $2\cdot 3^n -n$,
\item $2^{2^n}3^{2^m}$ is a Meertens number in base $2^{(2^n-n)}3^{2^m}-2^{m-n}$,
and 
\item $2^{3^n}3^{3^m}$ is a Meertens number in base $2^{3^n}3^{(3^m-n)}-3^{m-n}$.
\end{itemize}
\end{theorem}
\begin{proof}
Since $2n < 2\cdot 3^n$, $2\cdot 3^n$ is written as $1n$ in base $b=2\cdot 3^n -n$, and $M_b(2\cdot 3^n) = 2\cdot 3^n$. Similarly, $2^{n+1}\leq 2^{m+1} <  2^{(2^n-n)}3^{2^m}$ and 
 the $2$ digits  in the base $2^{(2^n-n)}3^{2^m}-2^{m-n}$ representation of $2^{2^n}3^{2^m}$ are $2^n$ and $2^m$ which is mapped by $M_b$ into $2^{2^n}3^{2^m}$. 
Next, $3^{n+1}\leq 3^{m+1} < 2^{3^n}3^{(3^m-n)}$ and 
 the $2$ digits  in the base $2^{3^n}3^{(3^m-n)}-3^{m-n}$ representation of $2^{3^n}3^{3^m}$ are $3^n$ and $3^m$ which is mapped by $M_b$ into $2^{3^n}3^{3^m}$ .
\eop
\end{proof}

\section{Injective G\"{o}del-like encodings}
As mentioned earlier, the encoding defined by $M_b(m)$ is not a proper G\"{o}del encoding as it is not one-to-one.  Next we look at some 
injective G\"{o}del-like encodings.
\subsection{$\alpha$-Meertens number} \label{sec:alphaM}
\begin{definition}
Let $b\geq 2$ and $0\leq d_i < b$ with $d_1 >0$ be integers such that $m = \sum_{i=1}^n d_i b^{n-i}$, then $N_b(m) = \prod_{i=1}^{n} p_i^{d_i+1}$
\end{definition}
Note that by the unique factorization theorem of the integers, $N_b$ is one-to-one on the set of positive integers.
We will call numbers such that $N_b(m) = m$ an {\em $\alpha$-Meertens number (in base $b$)}.  
Since the encoding is one-to-one, there cannot be a number $n$ that is a fixed point of this encoding in more than one base.  This is easily seen as a number will have different digits in different bases.
Some examples of $\alpha$-Meertens numbers in various bases are listed in Table \ref{tbl:m2number}.

\begin{table}[htbp]
\begin{center}
\begin{tabular}{c|c}
\hline\hline
base & $\alpha$-Meertens number \\
\hline
12 & 12, 24 \\
16 & 48 \\
24 & 96 \\
35 & 36 \\
64 & 384 \\
106 & 108 \\
107 & 216 \\
115 & 576 \\
192 & 1536 \\
321 & 324 \\
329 & 2304 \\
431 & 1296 \\
968 & 972 \\
970 & 1944 \\
1943 & 7776 \\
2048 & 24576 \\
2911 & 2916 \\
8742 & 8748 \\
8745 & 17496 \\
11662 & 34992 \\
24576 & 393216 \\
26237 & 26244 \\
46655 & 279936 \\
78724 & 78732 \\
78728 & 157464 \\
157462 & 629856 \\
236187 & 236196 \\
314925 & 944784 \\
\end{tabular}
\end{center}
\caption{$\alpha$-Meertens numbers in various bases.}\label{tbl:m2number}
\end{table}

The following result shows that there are an infinite number of $\alpha$-Meertens numbers.
\begin{theorem}
For $t\geq 0$, $3\cdot 2^{2^t+1}$ is an $\alpha$-Meertens number in base $b=3\cdot 2^{2^t-t+1}$.
\end{theorem}
\begin{proof}
First note that $2^t < 3\cdot 2^{2^t-t+1}$.  Then $3\cdot 2^{2^t+1}$ in base $3\cdot 2^{2^t-t+1}$ is the digit $2^t$ followed by the digit $0$ which maps to $3\cdot 2^{2^t+1}$ under the mapping $N_b$. \eop
\end{proof}

On the other hand, for a fixed $b$, there are only a finite number of $\alpha$-Meertens numbers in base $b$.

\begin{definition}
Let $p_i$ denote the $i$-th prime number, Let $p_n\#$ denote the primorial defined as $p_n\# = \prod_{i=1}^n p_i$.  Let $\vartheta(t)$ denote the first Chebyshev function defined as $\vartheta(n) = \sum_{p\leq n} \log(p)$ where $p$ ranges over all prime numbers less than or equal to $n$.
\end{definition}

\begin{theorem}\label{thm:primorial}
$p_n\# > n^{0.5972 n}$.  If $n\geq 947$, then $p_n\# > n^{0.980 n}$.
\end{theorem}
\begin{proof}
For $n=1$, the statement is trivially true. For $n>1$, note that $p_n\# = e^{\vartheta (p_n)}$.
Rosser \cite{rosser:prime:1939} showed that for $n\geq 1$, $p_n > n\log n$.  In \cite[Theorem 10]{rosser-shoenfeld:chebychev:1962}, it was shown that for $n\geq 7481$, $\vartheta(n) >
0.980 n$.  For primes $2 < p_n < 7481$, a simple computation shows that $\vartheta(p_n) > 0.5972 p_n$.
This implies that $p_n\# > e^{0.5972 p_n} > e^{0.5972 n \log n} = n^{0.5972 n}$ for $n> 1$.   The second part follows from the fact that the $947^{th}$ prime is 7481.  \eop
\end{proof}

\begin{lemma} \label{lem:m2bound}
If $m$ is an $\alpha$-Meertens number in base $b$ with $k$ digits, then $b^k > 2p_k\#$.
\end{lemma}
\begin{proof}
Since $m$ expressed in base $b$ has $k$ digits, $m< b^k$.  On the other hand, $m = N_b(m) \geq 2p_k\#$.
\eop
\end{proof}

\begin{theorem} \label{thm:m2bound2}
If $m$ is an $\alpha$-Meertens number in base $b$, then $m < b^{b^{1.675}}$.
\end{theorem}
\begin{proof}
Suppose that $m$ expressed as a base $b$ number has $k$ digits.  Then by Lemma \ref{lem:m2bound} and Theorem \ref{thm:primorial}, $b^k > 2p_k\# > k^{0.5972k}$, implying that $k < b^{1.675}$.  Thus
$m < b^k < b^{b^{1.675}}$.
\eop
\end{proof}

\begin{corollary} \label{cor:m2bound}
For a fixed $b$, let $k^*$ be the largest integer $k$ such that $b^k > 2p_k\#$.  Then $k^*\leq b^{1.675}$.  If $m$ is an $\alpha$-Meertens number in base $b$, then $m < b^{k^*}$.
\end{corollary}
\begin{proof}
This is a consequence of Lemma \ref{lem:m2bound} and Theorem \ref{thm:m2bound2}. \eop
\end{proof}

\begin{corollary} \label{cor:m2bound2}
For $b \leq 10000$, if $m$ is an $\alpha$-Meertens number in base $b$, then $m < b^{b-1}$.  If in addition  $608\leq b$, then $m < b^{\frac{b}{2}}$.
\end{corollary}
\begin{proof}
This requires a computer-assisted proof by computing the value of $k^*$ in Corollary \ref{cor:m2bound} for various $b$. \eop
\end{proof}

This allows us to improve Theorem  \ref{thm:m2bound2}.

\begin{theorem} \label{thm:m2bound3}
If $m$ is an $\alpha$-Meertens number in base $b$, then $m < b^{b^{1.02041}}$.
\end{theorem}
\begin{proof}
Suppose $m$ has $k$ digits in base $b$.  Then $m < b^k$.  If $k \geq 947$, then the proof of Theorem  \ref{thm:m2bound2}  combined with the second part of Theorem \ref{thm:primorial} shows that $k <b^{1.02041}$.
Suppose $k < 947$. If $b \geq 826$, then $b^{1.02041} \geq 947$ and thus again $k < b^{1.02041}$.  For $b < 826$, Corollary \ref{cor:m2bound2} shows that
$m < b^{b-1} < b^{b^{1.02041}}$.
\eop
\end{proof}

\begin{theorem} \label{thm:m2bound4}
If $m$ is an $\alpha$-Meertens number in base $b$, then $m < b^{b^{1+\epsilon}}$ where $\epsilon \rightarrow 0$ as $b\rightarrow\infty$.
\end{theorem}
\begin{proof}
It is well known that $\vartheta(x)$ behaves asympotically as $x$.  In particular, 
Ref. \cite[Theorem 4]{rosser-shoenfeld:chebychev:1962} shows that $\vartheta(x) > (1-\delta)x$ where $\delta\rightarrow 0$ as $x\rightarrow\infty$.  The rest of the proof is similar to the proof of Theorem \ref{thm:m2bound3} to show that $k < b^{\frac{1}{1-\delta}}$.\eop
\end{proof}

We conjecture that $k^*$ grows slower than the upper bound $b^{1.02041}$ or the asymptotic upper bound $b^{1+\epsilon}$ and that Corollary \ref{cor:m2bound2} is true for all $b$, i.e, all $\alpha$-Meertens numbers $m$ in base $b$ satisfies $m < b^{b-1}$ and satisfies
 $m < b^{\frac{b}{2}}$ for large enough $b$.  In particular, the first  few values of $k^*$ as a function of $b$ is shown in Table \ref{tbl:k-star} and a plot of $k^*$ versus $b$ is shown in Fig. \ref{fig:k-star} where $k^*$ appears to be less than $b$ for all $b$  and less than $\frac{b}{3}$ for large $b$.

\begin{table}[htbp]
\begin{center}
\begin{tabular}{|c||c|c|c|c|c|c|c|c|c|c|c|c|c|c|c|}
\hline
$b$ & $2$ & $3$ & $4$ & $5$  & $6$ & $7$ & $8$ & $9$  & $10$ & $11$ & $12$ & $13$  & $14$ & $15$ & $16$ \\
\hline
$k^*$ & $0$ & $0$ & $3$ & $4$  & $5$ & $6$ & $7$ & $8$  & $9$ & $10$ & $11$ & $12$  & $13$ & $14$ & $14$ \\ \hline
\end{tabular}
\end{center}
\caption{Values of $k^*$ as defined in Corollary \ref{cor:m2bound} for various $b$.}\label{tbl:k-star}
\end{table}

\begin{figure}[htbp]
\centerline{\includegraphics[width=7.2in]{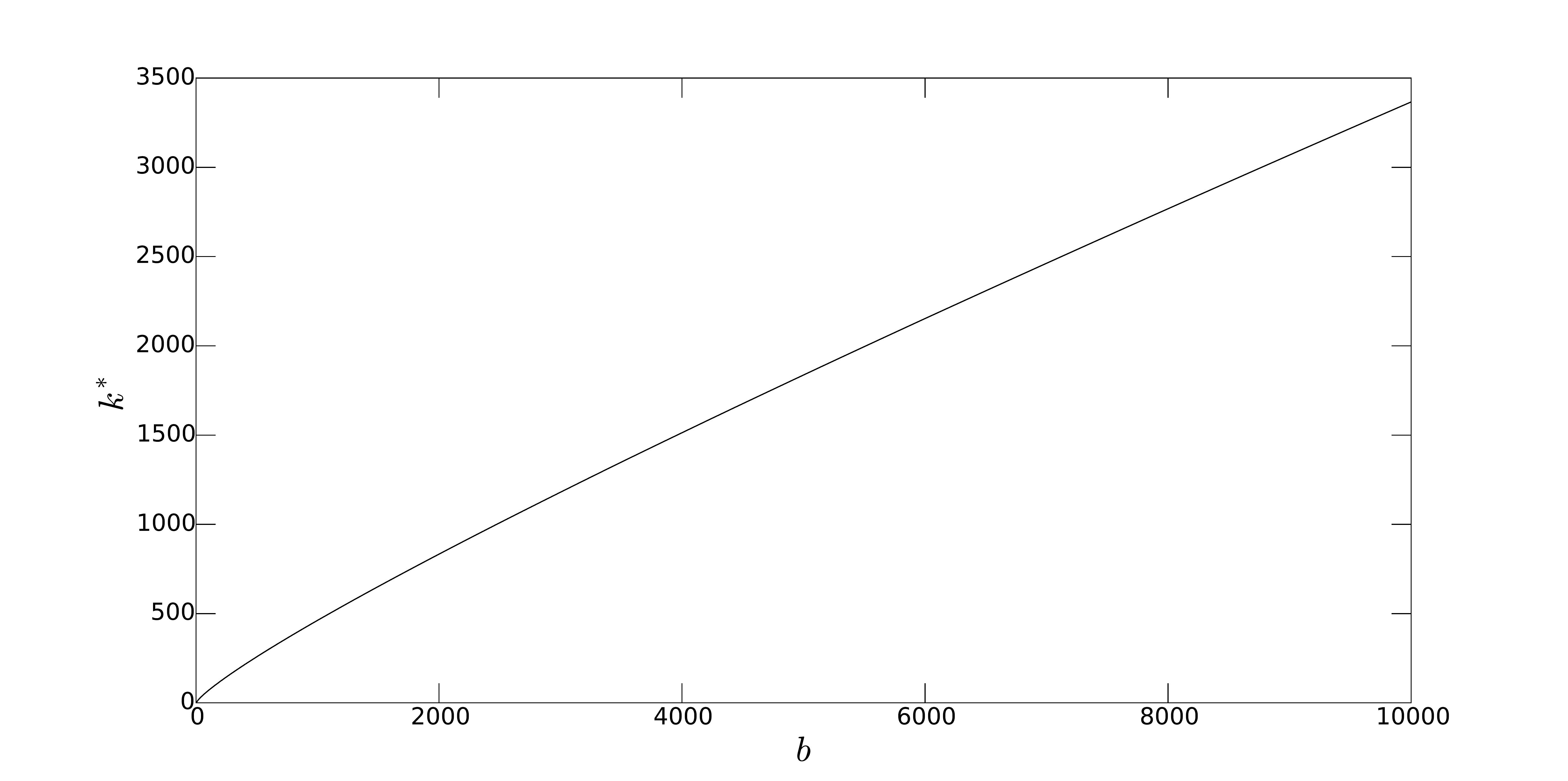}}
\caption{Plot of $k^*$ as defined in Corollary \ref{cor:m2bound} as a function of $b$.}
\label{fig:k-star}
\end{figure}

The following result shows that $12$ is the smallest base for which there exists an $\alpha$-Meertens number.
\begin{theorem}
There are no $\alpha$-Meertens numbers in base $b < 12$.
\end{theorem}
\begin{proof}
This again requires a computer-assisted proof.  If $m$ is an $\alpha$-Meertens number in base $b$, then Corollary \ref{cor:m2bound2} implies that $m < b^{b-1}$.  Next an exhausive search up to $b^{b-1}$ for $b < 12$ shows that there are no
 $\alpha$-Meertens numbers in base $b< 12$. \eop
\end{proof}

\subsection{Reverse Meertens number}
 Another way to define a one-to-one encoding is by reversing the digits and applying $M_b$, i.e. if the base-$b$ representation of a number $m$ is $d_n,\cdots ,d_1$, then the encoding
$M^r_b(m) = \prod_{i=1}^{n} p_i^{d_i}$ is one-to-one\footnote{Note that in contrast to the definition of $M_b$, there is not a requirement here that $d_n > 0$, i.e. leading zeros in the base-$b$ representation of $m$ do not affect the value of $M_b^r(m)$.} and we define a {\em reverse Meertens number} in base $b$ as a number $m$ such that $M^r_b(m) = m$.  As before, because this encoding is one-to-one, a number can be a reverse Meertens number in at most one number base.
In base $10$, $12 = 3^12^2$ is a reverse Meertens number.
Reverse Meertens numbers in different bases are listed in Table \ref{tbl:revnumber}.
\begin{table}[htbp]
\begin{center}
\begin{tabular}{c|c}
\hline\hline
base & reverse Meertens number\\
\hline
3 & 3, 10, 273 \\
5 & 6, 175 \\
7 & 100 \\
9 & 27 \\
10 & 12  \\
17 & 36  \\
21 & 24  \\
25 & 3125 \\
44 & 48 \\
49 & 823543  \\
70 & 144 \\
71 & 216 \\
91 & 96 \\
97 & 486 \\
186 & 192 \\
194 & 972 \\
285 & 576 \\
323 & 1296 \\
377 & 384 \\
574 & 1728 \\
760 & 768 \\
1148 & 2304 \\
1527 & 1536 \\
2187 & 19683 \\
2499 & 17496 \\
3062 & 3072 \\
4603 & 9216 \\
4605 & 13824 \\
5182 & 20736 \\
6133 & 6144 \\
7775 & 46656 \\
9997 & 69984 \\
12276 & 12288 \\
12440 & 62208\\
18426 & 36864\\
24563 & 24576\\
36860 & 110592 \\
49138 & 49152 \\
73721 & 147456 \\
98289 & 98304 \\
209951 & 1679616 \\
1119743 & 10077696 
\end{tabular}
\end{center}
\caption{Reverse Meertens numbers in various bases.}\label{tbl:revnumber}
\end{table}

Note that $17496$ is both a reverse Meertens number and an $\alpha$-Meertens number (in different bases).  Clearly, Meertens number such as 6, 100, 36 and 1296 which are palindromes in their respective bases (5, 7, 17 and 323) are also reverse Meertens numbers. 

\begin{theorem}\label{thm:23reverse}
For integers $m\geq n\geq 0$, 
\begin{itemize}
\item $3\cdot 2^n$ is a reverse Meertens number in base $3\cdot 2^n -n$, 
\item $2^{2^m}3^{2^n}$ is a reverse Meertens number in base $2^{(2^m-n)}3^{2^n}-2^{m-n}$
and 
\item $2^{3^m}3^{3^n}$ is a reverse Meertens number in base $2^{3^m}3^{(3^n-n)}-3^{m-n}$.
\end{itemize}
\end{theorem}
\begin{proof}
Since $2n < 3\cdot 2^n$, $3\cdot 2^n$ is written as $1n$ in base $b=3\cdot 2^n -n$, and $M^r_b(3\cdot 2^n) = 3\cdot 2^n$. Similarly, $2^{n+1}\leq 2^{m+1} < 2^{(2^m-n)}3^{2^n}$ and 
 the $2$ digits  in the base $2^{(2^m-n)}3^{2^n}-2^{m-n}$ representation of $2^{2^m}3^{2^n}$ are $2^n$ and $2^m$ which is mapped by $M_b^r$ into $2^{2^m}3^{2^n}$ .
Next, $3^{n+1}\leq 3^{m+1} < 3^{(3^n-n)}2^{3^m}$ and 
 the $2$ digits  in the base $2^{3^m}3^{(3^n-n)}-3^{m-n}$ representation of $2^{3^m}3^{3^n}$ are $3^n$ and $3^m$ which is mapped by $M_b^r$ into $2^{3^m}3^{3^n}$ .
\eop
\end{proof}

\begin{theorem} \label{thm:rmn}
$p_{r+1}^{p_{r+1}}$ is a reverse Meertens number in base $b = p_{r+1}^{\frac{p_{r+1}-1}{r}}$ if $r$ divides $p_{r+1}-1$.
\end{theorem}
\begin{proof}
Since $k< k^i$ for $k, i >1$, consider a base $k^i$ representation consisting of the digit $k$ followed by $r$ zeros, where $r = \frac{k-1}{i}$.
This represents the number $m = k(k^i)^r = k^{ir+1} = k^k$.  Under the mapping $M^r_b$, $M^r_b(m)$ = $p_{r+1}^k$.  Then the result follows if
$k = p_{r+1}$. \eop
\end{proof}

In particular, the first few primes $p_{r+1}$ satisfying the condition in Theorem \ref{thm:rmn} are:
$3$, $5$, $7$, $31$, $97$, $101$, $331$, $1009$, $1093$, $1117$, $1123$, $1129$, $3067$, $64621$, $480853$, etc.

\section{Zeroless Meertens numbers}
Next we study Meertens numbers in base $b$ without a zero digit when written in base $b$ representation.  We will call these numbers {\em zeroless Meertens numbers}.  Examples include 6, 18, 36, 96, 54, 216, 1296 with corresponding bases 5, 16, 17, 19, 51, 71, 323. 
Similarly, examples of zeroless reverse Meertens numbers are: 6, 12, 36, 24, 48, 144, 1296 with corresponding bases 5, 10, 17, 21, 44, 70, 323.
In fact, Theorems \ref{thm:23} and \ref{thm:23reverse}  show that there are an infinite number of bases with zeroless Meertens numbers or with zeroless reverse Meertens numbers.  On the other hand, for a fixed $b$, the number of zeroless Meertens numbers and zeroless reverse Meertens numbers is finite.

\begin{theorem} \label{thm:zeroless1}
If $b$ is squarefree, then a zeroless Meertens number  or a zeroless reverse Meertens number $m$ in base $b$ satisfies $m < b^{u-1}$, where $p_u$ is the largest prime dividing $b$.
\end{theorem}
\begin{proof}
Suppose $m$ is a zeroless Meertens number in base $b$.  Let $S$ be the set of indices of primes which divide $b$, i.e. $b = \prod_{i\in S} p_i$.
If $m$ has $u$ or more digits, then $d_i > 0$ for each $i \in S$, i.e. $b$ divides $m = \sum_{i} p_i^{d_i}$, and $m$ has a trailing zero digit in base $b$ leading to a contradiction.  The case of a zeroless reverse Meertens number is similar.
\eop
\end{proof}

The analysis in Section \ref{sec:alphaM}  can also be used to bound the number of zeroless (reverse) Meertens numbers.   
\begin{lemma}For a fixed $b$, let $l^*$ be the largest integer $l$ such that $b^l > p_k\#$.  Then $k^*\leq l^* \leq b^{1.675}$.
\end{lemma}
\begin{proof}
The proof is similar to the proof of Corollary \ref{cor:m2bound}. \eop
\end{proof}
\begin{theorem} \label{thm:zeroless2}
If $m$ is a zeroless Meertens number or a zeroless reverse Meertens number in base $b$, then $m < b^{l^*}\leq b^{b^{1.675}}$.
\end{theorem}
\begin{proof}
The proof is similar to the proof of Theorem \ref{thm:m2bound2}. \eop
\end{proof} 

\begin{table}[htbp]
\begin{center}
\begin{tabular}{|c||c|c|c|c|c|c|c|c|c|c|c|c|c|c|c|}
\hline
$b$ & $2$ & $3$ & $4$ & $5$  & $6$ & $7$ & $8$ & $9$  & $10$ & $11$ & $12$ & $13$  & $14$ & $15$ & $16$ \\
\hline
$l^*$ & $0$ & $2$ & $4$ & $5$  & $6$ & $7$ & $8$ & $9$  & $10$ & $11$ & $12$ & $12$  & $13$ & $14$ & $15$ \\ \hline
\end{tabular}
\end{center}
\caption{Values of $l^*$  for various $b$.}\label{tbl:l-star}
\end{table}

Using Theorems \ref{thm:zeroless1} and \ref{thm:zeroless2} and an exhaustive computer search we show that:
\begin{theorem}
\begin{itemize}
\item The number $6$ (associated with base $5$) is the only zeroless Meertens number for bases $< 12$.  
\item The numbers $6$ (associated with base $5$) and $12$ (associated with base $10$) are the only zeroless reverse Meertens number for bases $<12$.
\item There are no zeroless Meertens numbers or zeroless reverse Meertens numbers in bases $13$, $14$, or $15$.
\item The number $36$ is the only zeroless Meertens number and zeroless reverse Meertens number in base $17$.
\end{itemize}
\end{theorem}

We can also estimate the number of zero digits in a Meertens or reverse Meertens number:
\begin{theorem}
If $m$ is a Meertens or a reverse Meertens number in base $b$ with $u$ digits, then the number of zero digits in $m$ is larger than
\begin{equation}
u - e^{W(1.675u\log(b))}
\end{equation}
where $W$ is the Lambert W function.
\end{theorem}
\begin{proof}
Let $z$ be the number of zero digits in $m$.  Then
\begin{eqnarray*}
 b^u &>& m \geq p_{u-z}\# > (u-z)^{0.5972(u-z)}\\
u\log b &>& 0.5972(u-z)\log(u-z) \\
1.675 u\log b &>& (u-z)\log(u-z) \\
u-z &<& e^{W(1.675 u\log b)}
\end{eqnarray*}
\eop
\end{proof}

\section{Generalized Meertens numbers and generalized reverse Meertens numbers}
\begin{definition}
Given a pair of maps $f=\{f_1, f_2\}$ where $f_1:\mathbb{N}\rightarrow \mathbb{N}$, $f_2:\mathbb{N}\rightarrow \mathbb{N}$, 
define the map 
\[M^f(d_1,\cdots ,d_n) = \prod_{i=1}f_1(i)^{f_2(d_i)}.\]  A generalized Meertens number (GMN) in base $b$ is a number $m$ such that
$M^f(d_1,\cdots ,d_n) = m$ where $(d_1, \cdots, d_n)$ are the digits of $m$ in base $b$.  A generalized reverse Meertens number (GRMN) in base $b$ is a number 
$m$ such that $M^f(d_n,\cdots ,d_1) = m$.
\end{definition}

In the cases we discussed in the sections above, $f_1(i)$ is the $i$-th prime and $f_2(i) = i$ or $f_2(i) = i+1$.  For these cases, since $p^d > d$ for all primes $p$ and integers $d$, all GMN and GRMN in base $b$ must be larger or equal to $b$.  The tables above show that it it possible for a GMN or GRMN in base $b$ to be equal to $b$.  In particular, $2$ is a Meertens number in base $2$, $12$ is an $\alpha$-Meertens number in base $12$ and $3$ is a reverse Meertens number in base $3$.  
In fact, since
$b$ written in base $b$ is $10$, applying the digits $(1,0)$ (resp. the digits $(0,1)$) to $M^f$ will return a number $b$ which is a GMN (resp. GRMN) in base $b$.   This is summarized in the following result.

\begin{theorem}
Suppose $f_1(i) > i$ for all $i$.  
\begin{itemize}
\item If $c$ is a GMN or a GRMN in base $b$, then $c\geq b$.
\item If $f_1(1)^{f_2(1)}f_1(2)^{f_2(0)} > 1$, then $b$ is a GMN in base $b$ where $b=f_1(1)^{f_2(1)}f_1(2)^{f_2(0)}$.  
\item If $f_1(2)^{f_2(1)}f_1(1)^{f_2(0)} > 1$, then $b$ is a GRMN in base $b$ where $b=f_1(2)^{f_2(1)}f_1(1)^{f_2(0)}$.
\end{itemize}
\end{theorem}

Consider the case where $f_1$ and $f_2$ are both the identity map, i.e. $f_1(i) = f_2(i) = i$.  Clearly $1$ is a GMN and a GRMN in this case. 
In base $10$, $324 = 1^32^23^4$ is a GMN and $64 = 2^61^4$ is a GRMN. Table \ref{tbl:gmn} lists some GMN and GRMN numbers under these $f_i$'s.

\begin{table}[htbp]
\begin{center}
\begin{tabular}{c|c||c|c}
\hline\hline
base & generalized Meertens number & base & generalized reverse Meertens number\\
\hline
2 & 1350 & 2 & 2,6,12\\
4 & 108 & 3 & 120, 360\\
5 & 8 &  4 & 54\\
6 & 16 & 5 & 48\\
7 & 72 & 6 & 32\\
10 & 324 & 7 & 768\\
12 & 1458 & 8 & 216, 1728 \\
23 & 1728 & 10 & 64 \\
29 & 64 & 11 & 192, 729, 1536
\end{tabular}
\end{center}
\caption{Generalized Meertens numbers and generalized Meertens numbers in various bases for the case when $f_1$ and $f_2$ are identity maps. The number $1$ is omitted from this table.}\label{tbl:gmn}
\end{table}

\section{Conclusions}
We study Meertens numbers and their variations which are defined as fixed points of maps on the natural numbers.  Depending on the map, the set of such numbers can be sparse or abundant. We showed that for $\alpha$-Meertens numbers and zeroless (reverse) Meertens numbers these numbers are finite for a fixed base $b$.  It would be interesting to investigate whether this is true for the other variations as well and under what conditions.  Another open question is the asymptotic behavior of $k^*$ and $l^*$ as a function of $b$.

\bibliography{quant,markov,consensus,secure,synch,misc,stability,cml,algebraic_graph,graph_theory,control,optimization,adaptive,top_conjugacy,ckt_theory,math,number_theory,matrices,power,quantum}

\end{document}